\documentclass[12pt, letterpaper]{amsart}

\usepackage[margin=1in]{geometry}
\usepackage{amssymb}
\usepackage{mathtools}
\usepackage{enumerate}
\usepackage[colorlinks=true, linkcolor=blue, citecolor=blue, pagebackref=true]{hyperref}

 \usepackage{mathtools} \usepackage{etoolbox}
 \patchcmd{\section}{\scshape}{\bfseries}{}{} \makeatletter
 \renewcommand{\@secnumfont}{\bfseries} \makeatother

\newcommand{\bone}{\mathbf{1}}

\newcommand{\br}{\boldsymbol{r}}
\newcommand{\bt}{\boldsymbol{t}}

\newcommand{\calO}{\mathcal{O}}
\newcommand{\calS}{\mathcal{S}}

\newcommand{\ZZ}{\mathbb{Z}}
\newcommand{\E}{W}

\def\eps{{\varepsilon}}

\DeclarePairedDelimiter{\set}{\lbrace}{\rbrace}
\DeclarePairedDelimiter{\parens}{\lparen}{\rparen}
\DeclarePairedDelimiter{\brackets}{\lbrack}{\rbrack}
\DeclarePairedDelimiter{\abs}{\lvert}{\rvert}
\DeclarePairedDelimiter{\norm}{\lVert}{\rVert}

\theoremstyle{plain}
\newtheorem{theorem}{Theorem}
\theoremstyle{plain}
\newtheorem*{theorem*}{Theorem}
\theoremstyle{plain}

\theoremstyle{plain}
\newtheorem{claim}{Claim}
\theoremstyle{plain}
\newtheorem{lemma}[theorem]{Lemma}
\theoremstyle{plain}
\newtheorem{corollary}[theorem]{Corollary}
\theoremstyle{definition}
\newtheorem{question}[theorem]{Question}
\theoremstyle{definition}

\theoremstyle{definition}
\newtheorem{remark}[theorem]{Remark}
\theoremstyle{definition}
\newtheorem*{remark*}{Remark}
\theoremstyle{plain}
\newtheorem{conjecture}[theorem]{Conjecture}

\title[Counterexamples, covering systems, and zero-one
laws]{Counterexamples, covering systems, and zero-one laws for
  inhomogeneous approximation}
\author[F.~A.~Ram{\'i}rez]{Felipe A.~Ram{\'i}rez}
\thanks{The author was supported by the EPSRC Programme Grant
  EP/J018260/1.}
\address{Wesleyan University, Middletown, CT, USA}
\email{framirez@wesleyan.edu}

\begin{document}

\frenchspacing

\begin{abstract}
  We develop the inhomogeneous counterpart to some key aspects of the
  story of the Duffin--Schaeffer Conjecture~(1941). Specifically, we
  construct counterexamples to a number of candidates for a
  sans-monotonicity version of Sz{\"u}sz's inhomogeneous~(1958) version
  of Khintchine's Theorem~(1924). For example, given any real sequence
  $\set{y_i}$, we build a divergent series of non-negative reals
  $\psi(n)$ such that for any $y\in\set{y_i}$, almost no real number
  is inhomogeneously $\psi$-approximable with inhomogeneous parameter
  $y$. Furthermore, given any second sequence $\set{z_i}$ not
  intersecting the rational span of $\set{1,y_i}$, and assuming a
  dynamical version of Erd\H{o}s' Covering Systems Conjecture~(1950),
  we can ensure that almost every real number is inhomogeneously
  $\psi$-approximable with any inhomogeneous parameter
  $z\in\set{z_i}$. Next, we prove a positive result that is near
  optimal in view of the limitations that our counterexamples
  impose. This leads to a discussion of natural analogues of the
  Duffin--Schaeffer Conjecture and Duffin--Schaeffer Theorem~(1941) in
  the inhomogeneous setting.  As a step toward these, we prove
  versions of Gallagher's Zero-One Law~(1961) for inhomogeneous
  approximation by reduced fractions.
\end{abstract}

\maketitle

{\footnotesize \tableofcontents}

\section{Introduction and results}
\label{sec:introduction-history}

The basic question in (homogeneous) Diophantine approximation is about
approximating real numbers by rational numbers. Given a real number
$x$, how small can $\abs{x - a/n}$ be as a function of $n$, where
$a/n$ is rational? This is the same as asking how small we can make
$\norm{nx}$, the distance from $nx$ to the nearest integer. In
\emph{inhomogeneous} Diophantine approximation we have some other real
number $y$---our inhomogeneous parameter---and we try to minimize
$\norm{nx+y}$.

There are of course innumerable questions one can ask for homogeneous
and inhomogeneous approximation. It often happens that the homogeneous
theory is mirrored in the inhomogeneous setting. For an example of
this, take the homogeneous and inhomogeneous versions of Khinthine's
Theorem, below. Sometimes, techniques for homogeneous approximation
can lead to results for inhomogeneous
approximation. (See, for example,~\cite[Chapter V]{CasselsintrotoDA}.) The reverse
may also happen: inhomogeneous considerations can illuminate facts in
the homogeneous world. A famous example of this is Kurzweil's
Theorem~\cite{Kurzweil}, where badly (homogeneously) approximable real
numbers $x$ are characterized by their behavior with respect to all
possible inhomogeneous expressions $\norm{nx+y}$.

Our goal here is to develop the inhomogeneous counterpart to some of
the narrative underlying one of Diophantine approximation's most
vexing open problems: the Duffin--Schaeffer Conjecture.

\subsection{Homogeneous theory}
\label{sec:homothy}

The following theorem is the foundation of \emph{metric Diophantine
  approximation}. Note that we use the term `approximating function'
to mean a non-negative real-valued function of the natural numbers.

\theoremstyle{plain} \newtheorem*{kt}{Khintchine's Theorem}
\begin{kt}[\cite{Khintchineonedimensional}]
  Let $\psi$ be a non-increasing approximating function. Then almost
  every or almost no real number $x$ satisfies the inequality
  $\norm{nx}<\psi(n)$ with infinitely many integers $n$, according as
  $\sum_n\psi(n)$ diverges or converges.
\end{kt}

A great deal of effort has been devoted to finding a suitable
Khintchine-like statement that would not require the approximating
function to be monotonic. In 1941, Duffin and
Schaeffer~\cite{DuffinSchaeffer} showed that one cannot simply remove
the word `non-increasing' from Khintchine's Theorem. Specifically,
they produced an approximating function $\psi$ that would serve as
counterexample to the divergence part of the resulting
statement. (Theorems~\ref{thm:sequencecounterex}/\ref{thm:moreover}
and~\ref{thm:cantorex} generalize this to the inhomogeneous setting.)
Still, the \emph{Duffin--Schaeffer Counterexample} has the property
that $\sum_n\varphi(n)\psi(n)/n$ converges, where $\varphi$ is Euler's
$\varphi$-function. This led to their formulating what is now one
of the foremost open problems of Diophantine approximation.

\theoremstyle{plain} \newtheorem*{dsc}{Duffin--Schaeffer Conjecture}
\begin{dsc}[\cite{DuffinSchaeffer}] If $\psi$ is an approximating
  function such that the sum $\sum_n\varphi(n)\psi(n)/n$ diverges,
  then almost every $x$ satisfies the inequality $\abs{nx-a}<\psi(n)$
  with infinitely many coprime\footnote{One might
    also consider a version of this conjecture where the word
    `coprime' has been removed. This is of course weaker than the
    stated conjecture. Whether it is easier to prove seems to be an
    unexplored question.} integer pairs $(a,n)$.
\end{dsc}

The Duffin--Schaeffer Conjecture continues to be actively pursued, and
has only been neared by partial and related results. In the original
paper, Duffin and Schaeffer proved the first such partial result.

\theoremstyle{plain} \newtheorem*{dst}{Duffin--Schaeffer Theorem}
\begin{dst}[\cite{DuffinSchaeffer}]
  If $\psi$ is an approximating function such that $\sum_n\psi(n)$
  diverges and
  \begin{equation*}
    \limsup_{N\to\infty}\parens*{\sum_{n=1}^N \frac{\varphi(n)\psi(n)}{n}}\parens*{\sum_{n=1}^N\psi(n)}^{-1}>0,
  \end{equation*}
  then almost every $x$ satisfies the inequality $\abs{nx-a}<\psi(n)$
  with infinitely many coprime integer pairs $(a,n)$.
\end{dst}

Many others have followed. For example, Erd\H{o}s verified the
conjecture for approximating functions that take the value $\eps/n$ on
their support, for some $\eps>0$~\cite{ErdosDS}; Vaaler improved this
to functions of the form $\psi(n)=O(n^{-1})$~\cite{Vaaler}; Pollington
and Vaughan proved that the Duffin--Schaeffer Conjecture holds in
higher dimensions~\cite{PollingtonVaughan}; and recently, Beresnevich,
Haynes, Harman, Pollington, and Velani have proved the conjecture
under ``extra divergence'' assumptions, and Aistleitner has proved it
for ``slow divergence''~\cite{HPVextra, BHHVextraii, AistleitnerDS}.

An important tool for attacks on the Duffin--Schaeffer Conjecture, and
indeed in many other problems in Diophantine approximation, is the
``zero-one law''---a statement precluding that the measure of a set be
anything other than zero or one. The following zero-one law tells us
that the set that is predicted to be full in the Duffin--Schaeffer
Conjecture is either null or full.

\theoremstyle{plain} \newtheorem*{gzol}{Gallagher's Zero-One Law}
\begin{gzol}[\cite{Gallagher01}]
  Let $\psi$ be an approximating function. Then almost every or almost
  no $x$ satisfies the inequality $\abs{nx-a}<\psi(n)$ with infinitely
  many coprime integer pairs $(a,n)$.
\end{gzol}

Modifications of Gallagher's proof yield Theorems~\ref{thm:E01}
and~\ref{thm:eitheror}---inhomogeneous versions where we consider
countably many inhomogeneous parameters simultaneously.

\subsection{Inhomogeneous theory}
\label{sec:inhomothy}

It was later proved by Sz{\"u}sz that an \emph{inhomogeneous} version
of Khintchine's Theorem also holds.

\theoremstyle{plain} \newtheorem*{ikt}{Inhomogeneous Khintchine
  Theorem}
\begin{ikt}[\cite{SzuszinhomKT}]
  Let $y$ be a real number and $\psi$ a non-increasing approximating
  function. Then almost every or almost no real number $x$ satisfies
  the inequality $\norm{nx+y}<\psi(n)$ with infinitely many integers
  $n$, according as $\sum_n\psi(n)$ diverges or converges.
\end{ikt}

Again, one can ask about the possibility of removing the word
`non-increasing' from this theorem. Our first result,
Theorem~\ref{thm:sequencecounterex}, shows that this is impossible by
giving a counterexample to the resulting statement. In fact, it is a
counterexample to much more.

The following theorem shows that if one lets the inhomogeneous
parameter vary, then monotonicity \emph{can} be removed.

\theoremstyle{plain} \newtheorem*{dmkt}{Doubly Metric Inhomogeneous
  Khintchine Theorem}
\begin{dmkt}[\cite{CasselsintrotoDA}]
  Let $\psi$ be an approximating function. Then almost every or almost
  no real pair $(x,y)$ satisfies the inequality $\norm{nx+y}<\psi(n)$
  with infinitely many integers $n$, according as $\sum_n\psi(n)$
  diverges or converges.
\end{dmkt}

One may therefore be tempted to suspect that if we restrict the
inhomogeneous parameter only slightly (instead of letting it vary over
all real numbers as in the above theorem) then we will still retain a
similar statement to the one above, without the need for
monotonicity. For example, we may require the inhomogeneous parameter
to lie on an equidistributed sequence of real numbers, and hope that a
``doubly metric''-style statement will hold for a ``density-one''
subsequence (the idea being that this equidistributed sequence would
be a generic sampling from the full-measure set provided by the Doubly
Metric Inhomogeneous Khintchine Theorem). But this is also ruled out
by our first result. We construct counterexamples for \emph{any} given
sequence of inhomogeneous parameters.

\begin{theorem}\label{thm:sequencecounterex}
  For any sequence $\set{y_i}$ of real numbers there is an
  approximating function $\psi$ such that the sum $\sum_n\psi(n)$
  diverges and such that for any $y\in\set{y_i}$, there are
  \textbf{almost no} real numbers $x$ for which the inequality
  $\norm{nx+ y}< \psi(n)$ is satisfied by infinitely many integers
  $n\geq 1$.
\end{theorem}

In fact, the counterexamples we construct are automatically
counterexamples for $y=0$, as well as \emph{any} rational combination
of $1$ with finitely many elements of the sequence $\set{y_i}$, which
leads to the question of whether this can be avoided. (See
Remark~\ref{rem:moreover} and
Questions~\ref{q:ratcomb},~\ref{q:hominhom}, and~\ref{q:noyes}.) As
for inhomogeneous parameters that are \emph{not} in the rational span
of $1$ with $\set{y_i}$, we have the following continuation of
Theorem~\ref{thm:sequencecounterex}.

\begin{theorem}[To be read as a continuation of Theorem~\ref{thm:sequencecounterex}]\label{thm:moreover}
  Moreover, if Conjecture~\ref{lem:conjlem} is true, and $\set{z_i}$
  is a second sequence of real numbers no element of which lies in the
  rational span of $1$ with finitely many elements of $\set{y_i}$,
  then we may take $\psi$ to have the additional property that for any
  $z\in\set{z_i}$, \textbf{almost every} real number $x$ satisfies the
  inequality $\norm{nx+ z}< \psi(n)$ with infinitely many $n\geq 1$.
\end{theorem}

\begin{remark*}
  Conjecture~\ref{lem:conjlem} is a dynamical version of Erd\H{o}s'
  famous ``covering systems'' conjecture~\cite{Erdoscongruences,
    Erdos2} and recent progress on it~\cite{FFKPY, Hough}. We have
  left a full discussion of this for~\S\ref{sec:proof-theor-refthm:m}
  so as not to disrupt the present narration. Suffice it to say
  that~\cite{FFKPY, Hough} provide overwhelming evidence in favor of
  Conjecture~\ref{lem:conjlem}.

  Notice that by the Borel--Cantelli Lemma, the conclusion of
  Theorem~\ref{thm:moreover} forces $\sum_n \psi(n)$ to diverge, and
  so when Theorems~\ref{thm:sequencecounterex} and~\ref{thm:moreover}
  are read together the divergence condition is redundant. We list and
  prove the two theorems separately, both for the readers'
  convenience, and because Theorem~\ref{thm:sequencecounterex} has a
  simpler proof that does not require Conjecture~\ref{lem:conjlem}.
\end{remark*}

It is natural to ask whether points in the rational span of
$\set{1,y_i}$ are the \emph{only} ones for which the counterexamples
in Theorem~\ref{thm:sequencecounterex}/\ref{thm:moreover} work. Rather
than answer this directly, we construct counterexamples that work for
an \emph{uncountable} set of possible inhomogeneous parameters (as
well as their rational combinations with $1$).

\begin{theorem}\label{thm:cantorex}
  There is an uncountable ``Cantor-type'' set $C$ of real numbers and
  an approximating function $\psi$ such that the sum $\sum_n\psi(n)$
  diverges and such that for any $y\in C$ there are almost no real
  numbers $x$ for which the inequality $\norm{nx+y}< \psi(n)$ is
  satisfied by infinitely many integers $n\geq 1$.
\end{theorem}

\begin{remark*}
  From the construction in the proof, it will be clear that there are
  uncountably many such Cantor sets, each with its own approximating
  function.
\end{remark*}

Theorems~\ref{thm:sequencecounterex} and~\ref{thm:cantorex} imply that
a non-monotonic inhomogeneous version of Khintchine's Theorem has
absolutely no hope of being true, even if we relax the requirement
that the inhomogeneous part be fixed and instead let it come from some
set of \emph{permitted} real numbers. Still, if we take the permitted
set to be an equidistributed sequence, then we can prove the
following.

\begin{theorem}\label{thm:arbitrary}
  Let $\set{y_m}$ be an equidistributed sequence $\bmod\; 1$. Suppose
  $\psi$ is an approximating function such that $\sum_n\psi(n)$
  diverges. Then for every $R,\eps>0$ there is a density-one set of
  integers $m\geq1$ with the property that the set of real $x$ for
  which the inequality $\norm*{nx +y_m}<\psi(n)$ has at least $R$
  integer solutions $n$, has measure at least $1-\eps$.
\end{theorem}

Theorem~\ref{thm:sequencecounterex} prevents us from doing too much
better than this. On the other hand, the approximating functions
$\psi$ in Theorems~\ref{thm:sequencecounterex} and~\ref{thm:cantorex}
have the property that $\sum_n\varphi(n)\psi(n)/n$ converges. So it
becomes natural to ask about inhomogeneous versions of the
Duffin--Schaeffer Conjecture, where we seek approximations by
\emph{reduced} fractions, and use an accordingly modified divergence
condition. There is, of course, the direct inhomogeneous translation:

\theoremstyle{plain} \newtheorem*{idsc}{Inhomogeneous
  Duffin--Schaeffer Conjecture}
\begin{idsc}
  Let $y$ be a real number. If $\psi$ is an approximating function
  such that $\sum_n\varphi(n)\psi(n)/n$ diverges, then for almost
  every real number $x$ there are infinitely many coprime integer
  pairs $(a,n)$ such that $\abs{nx-a+y}<\psi(n)$.
\end{idsc}

But this is much stronger than the original Duffin--Schaeffer
Conjecture, and probably therefore harder. It makes sense to explore
related questions. For example, in the spirit of
Questions~\ref{q:ratcomb} and~\ref{q:hominhom}, are there any
dependencies between the Inhomogeneous Duffin--Schaeffer Conjecture
for one inhomogeneous parameter versus another? And, in the spirit of
Theorems~\ref{thm:sequencecounterex}--\ref{thm:arbitrary}, can we make
progress on an inhomogeneous version of the Duffin--Schaeffer
Conjecture where we do not fix the inhomogeneous parameter, but
instead let it come from some predetermined sequence? We discuss these
questions in~\S\ref{sec:inhom-duff-scha}, and
in~\S\ref{sec:inhom-zero-one} we make some progress in the form of the
following zero-one laws, inspired by Gallagher's and deduced by
modifying his proof.

\begin{theorem}\label{thm:E01}
  Let $y$ be a real number and $\psi$ an approximating function.
  \begin{itemize}
  \item For almost every or almost no real number $x$ there exists an
    integer $m\geq 1$ such that $\abs{nx - a + my}<\psi(n)$ has
    infinitely many coprime integer solutions $(a,n)$.
  \item For almost every or almost no real number $x$ there exist
    \emph{infinitely many} such integers $m\geq 1$.
  \end{itemize}
\end{theorem}

\begin{remark*}
  Notice that Gallagher's Zero-One Law~\cite{Gallagher01} is the same
  statement when $y=0$. (Of course, in that case, the integer $m$ has
  no role to play.) Our proof follows his.
\end{remark*}

\begin{theorem}\label{thm:eitheror}
  Let $y$ be a real number and $\psi$ an approximating function.
  \begin{itemize}
  \item \textbf{Either} for every integer $m\geq 1$, there are almost
    no real $x$ for which
    \begin{equation}
      \abs{nx - a + my}<\psi(n) \textrm{ for infinitely many }(a,n)=1,\label{eq:coprime}
    \end{equation}
  \item \textbf{Or} at least one of the following
    holds:\begin{itemize}
    \item There is some $m\geq 1$ such that~\eqref{eq:coprime} holds
      for almost every real $x$.
    \item For any $\eps>0$ there are arbitrarily many $m\geq 1$ such
      that the set of $x$ for which~\eqref{eq:coprime} holds has
      measure greater than $1-\eps$.
    \end{itemize}
  \end{itemize}
\end{theorem}

These leave open the more obvious zero-one law, where we have a single
fixed inhomogeneous parameter. We will visit this in a future project.

\section{Proofs of Theorems~\ref{thm:sequencecounterex}
  and~\ref{thm:cantorex}: Counterexamples}
\label{sec:midscounterex}

In this section we prove Theorems~\ref{thm:sequencecounterex}
and~\ref{thm:cantorex}. We begin with the following lemma.

\begin{lemma}\label{lem:multilem}
  For any $\eps>0$, integer $\ell\geq 1$, and real numbers
  $y_1, \dots, y_\ell$, the sum
  \begin{equation*}
    \sum_{\max_i \norm{ny_i}<\eps}\frac{1}{n+1}
  \end{equation*}
  diverges.
\end{lemma}

\begin{proof}[\textbf{Proof}]
  The numbers $n\geq 1$ such that $\max_i\norm{ny_i}<\eps$ comprise a
  set of positive density in the natural numbers. Therefore, the sum
  of reciprocals diverges.
\end{proof}

\begin{proof}[\textbf{Proof of Theorem~\ref{thm:sequencecounterex}}]
  Let $\set{y_i}$ be a sequence of real numbers. Let
  $\set{n_m^{(1)}}_m$ be the sequence of times when
  $\norm{n_m^{(1)} y_1}\leq 2^{-2}$, and let
  \begin{equation*}
    K_1 = \prod_{m=1}^{M_1} \parens*{n_m^{(1)}+1}
  \end{equation*}
  where $M_1>0$ is chosen so that
  \begin{equation*}
    \sum_{m=1}^{M_1} \frac{1}{n_m^{(1)}+1} \geq 2^2.
  \end{equation*}
  Such an $M_1$ exists by Lemma~\ref{lem:multilem}.
  
  Inductively, let $\set{n_m^{(j)}}_m$ be the sequence of times
  $> K_{j-1}$ when
  \begin{equation}
    \max_{i=1, \dots, j}\norm{n_m^{(j)} y_i}\leq 2^{-j-1},\label{eq:1}
  \end{equation}
  and let
  \begin{equation*}
    K_j = \prod_{m=1}^{M_j} \parens*{n_m^{(j)}+1}
  \end{equation*}
  where $M_j>0$ is chosen so that
  \begin{equation}\label{eq:Mj}
    \sum_{m=1}^{M_j} \frac{1}{n_m^{(j)}+1} \geq 2^{j+1}.
  \end{equation}
  Again, Lemma~\ref{lem:multilem} allows us to choose such an $M_j$.
  
  Let $k_m^{(j)} = K_j/(n_m^{(j)}+1)$. Notice that the $k_m^{(j)}$ are
  pairwise distinct because the $K_j$ form a strictly increasing sequence of
  positive integers, and for any fixed $j$, we have
  \begin{equation*}
    K_{j-1} < k_{M_j}^{(j)} < k_{M_j-1}^{(j)} < \dots < k_{2}^{(j)} <
    k_{1}^{(j)} < K_j,
  \end{equation*}
  by our construction.

  Define
  \begin{equation}\label{eq:def}
    \psi(k) =
    \begin{dcases}
      \frac{k}{K_j}2^{-j-1} =\frac{2^{-j-1}}{n_m^{(j)}+1}
      &\textrm{for } k=k_m^{(j)},  m=1, \dots, M_j \\
      0 &\textrm{otherwise.}
    \end{dcases}
  \end{equation}
  Let $y\in\set{y_i}$, and let us set the notation
  \begin{equation*}
    E_n^y(\psi) = \frac{\ZZ+y}{n} + \parens*{-\frac{\psi(n)}{n}, \frac{\psi(n)}{n}}
  \end{equation*}
  and
  \begin{equation*}
    E_n^y(\eps) = \frac{\ZZ+y}{n} + \parens*{-\frac{\eps}{n}, \frac{\eps}{n}}
  \end{equation*}
  when the argument is a real constant.

  The rest of this proof follows from three claims that are proved
  separately below, because they will be used again
  later. Claim~\ref{cl:containment} states that if $y_j$ comes after
  $y\in\set{y_i}$, then
  $E_{k_m^{(j)}}^y(\psi) \subset E_{K_j}^y(2^{-j})$ for all
  $m=1, \dots, M_j$. This implies that
  \begin{equation*}
    \limsup_{n\to\infty}E_n^y(\psi) \subset
    \limsup_{j\to\infty}E_{K_j}^y(2^{-j}).
  \end{equation*}
  Claim~\ref{cl:limsup} shows that
  \begin{equation*}
    \abs*{\limsup_{j\to\infty}E_{K_j}^y(2^{-j})}=0.
  \end{equation*}
  And Claim~\ref{cl:div} shows that $\sum_n \psi(n)=\infty$. This
  proves the theorem.
\end{proof}

\begin{claim}\label{cl:containment}
  If $y_j$ comes after $y\in\set{y_i}$, then
  $E_{k_m^{(j)}}^y(\psi)\subset E_{K_j}^y(2^{-j})$ for all
  $m=1, \dots, M_j$.
\end{claim}

\begin{proof}[\textbf{Proof}]
  First, we show that
  \begin{equation*}
    \frac{\ZZ + y}{k_m^{(j)}} \subset \frac{\ZZ+y}{K_j}
    + \brackets*{-\frac{2^{-j-1}}{K_j}, \frac{2^{-j-1}}{K_j}}.
  \end{equation*}
  For this, it is enough to show that $y/k_m^{(j)}$ is within a
  distance of $2^{-j-1}/K_j$ from an element $(\ell + y)/K_j$ of
  $(\ZZ+y)/K_j$. Then all elements of $(\ZZ+y)/k_m^{(j)}$ will also
  be, because they will just be shifts of $y/k_m^{(j)}$ by integer
  multiples of $1/k_m^{(j)}$, which are of course integer multiples of
  $1/K_j$. But
  \begin{equation*}
    \min_{\ell} \abs*{\frac{y}{k_m^{(j)}} - \frac{\ell +
        y}{K_j}} = \frac{1}{K_j}\min_{\ell} \abs*{n_m^{(j)}y - \ell}
    \leq \frac{2^{-j-1}}{K_j},
  \end{equation*}
  which proves it.

  Now we have
  \begin{multline*}
    E_{k_m^{(j)}}^y(\psi) = \frac{\ZZ+y}{k_m^{(j)}}
    + \parens*{-\frac{\psi(k_m^{(j)})}{k_m^{(j)}},\frac{\psi(k_m^{(j)})}{k_m^{(j)}}}\\
    \subset\frac{\ZZ+y}{K_j}
    + \parens*{-\frac{\psi(k_m^{(j)})}{k_m^{(j)}} -
      \frac{2^{-j-1}}{K_j},\frac{\psi(k_m^{(j)})}{k_m^{(j)}}+\frac{2^{-j-1}}{K_j}}\\
    =\frac{\ZZ+y}{K_j}
    + \parens*{-\frac{2^{-j}}{K_j},\frac{2^{-j}}{K_j}} =
    E_{K_j}^y(2^{-j}),
  \end{multline*}
  which is what we wanted to prove.
\end{proof}

\begin{claim}\label{cl:limsup}
  $\abs{\limsup_{j\to\infty}E_{K_j}^y(2^{-j})} = 0.$
\end{claim}

\begin{proof}[\textbf{Proof}]
  The sum
  \begin{equation*}
    \sum_j \abs*{E_{K_j}^y(2^{-j})\cap[0,1)} = \sum_j 2^{-j+1}
  \end{equation*}
  converges, therefore by the Borel--Cantelli Lemma we have
  $\abs{\limsup_{j\to\infty}E_{K_j}^y(2^{-j})\cap[0,1)} = 0$. The
  claim follows because $\limsup_{j\to\infty}E_{K_j}^y(2^{-j})$ is
  $1$-periodic.
\end{proof}

\begin{claim}\label{cl:div}
  $\sum_n\psi(n) = \infty.$
\end{claim}

\begin{proof}[\textbf{Proof}]
  We compute the sum:
  \begin{align*}
    \sum_n \psi(n) &= \sum_j \sum_{m=1}^{M_j} \psi(k_m^{(j)}) \\
                   &= \sum_j \sum_m \frac{k_m^{(j)}}{K_j}2^{-j-1}\\
                   &= \sum_j2^{-j-1}\sum_{m=1}^{M_j} \frac{1}{n_m^{(j)} + 1} \\
                   &\overset{(\ref{eq:Mj})}{\geq} \sum_j 1,
  \end{align*}
  which diverges.
\end{proof}

\begin{proof}[\textbf{Proof of Theorem~\ref{thm:cantorex}}]
  Let $\set{n_m^{(0)}}_{m=1}^{M_0}$ be such that
  \begin{equation*}
    \sum_{m=1}^{M_0} \frac{1}{n_m^{(0)}+1} \geq 2
  \end{equation*}
  and let
  \begin{equation*}
    K_0 = \prod_{m=1}^{M_0} \parens*{n_m^{(0)}+1}.
  \end{equation*}
  Let $C_0$ be the set of all $y\in[0,1]$ such that
  $\norm{n_m^{(0)}y}\leq 2^{-1}$ for $m=1, \dots, M_0$. Then
  $C_0=\bigcup_{i=1}^{i_0}I_i^{(0)}$ is a finite union of subintervals
  of $[0,1]$. (Actually, since $\norm{\cdot}$ \emph{always} takes
  values $\leq 2^{-1}$, we have $C_0=[0,1]$. This $0$th step is really
  only here to seed the inductive process.)
  
  We proceed inductively. For $j\geq 1$, let $L_j > K_{j-1}$ be such
  that $[0,1)\subset L_jI_i^{(j-1)}$ for all $i=1,\dots, i_{j-1}$.
  Pick an integer $\ell_j \geq 1$ and real numbers
  $y_1^{(j)}, \dots, y_{\ell_j}^{(j)}\in C_{j-1}$ and let
  $\set{n_m^{(j)}}_{m=1}^{M_j}$ be such that $n_m^{(j)}\geq L_j$,
  $\max_i\norm{n_m^{(j)}y_i^{(j)}}\leq 2^{-j-1}$, and
  \begin{equation}\label{eq:Mj}
    \sum_{m=1}^{M_j} \frac{1}{n_m^{(j)}+1} \geq 2^{j+1}.
  \end{equation}
  It is guaranteed that we can do this by Lemma~\ref{lem:multilem}.
  Let
  \begin{equation*}
    K_j = \prod_{m=1}^{M_j} \parens*{n_m^{(j)}+1}.
  \end{equation*}
  Let $C_j$ be the set of $y\in C_{j-1}$ with
  $\norm{n_m^{(j)}y} \leq 2^{-j-1}$ for all $m=1, \dots, M_j$. Then
  $C_j=\bigcup_{i=1}^{i_j}I_i^{(j)}$ is a finite union of closed
  subintervals of $C_{j-1}$. (These subintervals are non-empty because
  we have constructed them to contain
  $y_1^{(j)}, \dots, y_{\ell_j}^{(j)}$. That they are \emph{contained}
  in $C_{j-1}$ is ensured by the condition $n_m^{(j)}\geq L_j$.) Our
  uncountable ``Cantor-type'' set is $C = \bigcap_{j\geq 0}C_j$.

  Let $k_m^{(j)} = K_j/(n_m^{(j)}+1)$. Define $\psi$
  by~(\ref{eq:def}). Let $y\in C$, and define $E_n^y(\psi)$ and
  $E_{K_j}^y(2^{-j})$ as in the proof of
  Theorem~\ref{thm:sequencecounterex}.

  A simple modification of Claim~\ref{cl:containment} shows that
  \begin{equation*}
    \limsup_{n\to\infty}E_n^y(\psi) \subset
    \limsup_{j\to\infty}E_{K_j}^y(2^{-j}),
  \end{equation*}
  Claim~\ref{cl:limsup} shows that
  \begin{equation*}
    \abs*{\limsup_{j\to\infty}E_{K_j}^y\parens*{2^{-j}}}=0,
  \end{equation*}
  and Claim~\ref{cl:div} shows that $\sum_n \psi(n)=\infty$. This
  proves the theorem.
\end{proof}

\begin{remark}\label{rem:moreover}
  Notice that the counterexamples in
  Theorems~\ref{thm:sequencecounterex} and~\ref{thm:cantorex} also
  serve as counterexamples for the homogeneous case, because
  Claim~\ref{cl:containment} holds for $y=0$. More strikingly, they
  work for any rational combination of $1$ with elements of the
  sequence $\set{y_i}$ (in the case of
  Theorem~\ref{thm:sequencecounterex}, and the set $C$ in the case of
  Theorem~\ref{thm:cantorex}), as we will now sketch.

  Suppose $w_1, \dots, w_\ell\in\set{y_i}$ and let
  $z = a_1 w_1 + \dots + a_\ell w_\ell + b$, were
  $a_1, \dots, a_\ell, b$ are rational. We leave as an exercise to
  show that there exist positive integers $c_1, c_2$ depending on
  $a_1, \dots, a_\ell, b$ such that
  \begin{equation*}
    E_{k_m^{(j)}}^z(\psi) \subset E_{c_1 K_j}^z\parens*{c_2 2^{-j}}
  \end{equation*}
  for all sufficiently large $j$ and $m=1, \dots, M_j$, where $\psi$
  is the counterexample constructed in the proof of
  Theorem~\ref{thm:sequencecounterex}. (One can take $c_1$ as a common
  denominator for $a_1, \dots, a_\ell, b$, while $c_2$ can be taken
  sufficiently large.) Therefore,
  \begin{equation*}
    \limsup_{n\to\infty}E_{n}^z(\psi) \subset \limsup_{j\to\infty} E_{c_1 K_j}^z\parens*{c_2 2^{-j}}.
  \end{equation*}
  But an application of the Borel--Cantelli Lemma shows that
  \begin{equation*}
    \limsup_{j\to\infty} E_{c_1 K_j}^z\parens*{c_2 2^{-j}}
  \end{equation*}
  has measure zero.
\end{remark}

What is striking in the above remark is not the fact that one can find
a counterexample for all of these rational combinations. (After all,
Theorem~\ref{thm:sequencecounterex} already guarantees this, since the
rational span of $\set{1, y_i}$ is itself a denumerable set.) It is
the fact that all these rational combinations come for free from the
construction, without having to enumerate them as a new sequence. This
leads us to ask whether this is just a side-effect of our particular
construction, or if it is unavoidable.

\begin{question}\label{q:ratcomb}
  Suppose $\psi$ is an approximating function such that for any $y$
  among a fixed set $\set{y_1, \dots, y_\ell}$ of real numbers, almost
  no real number $x$ satisfies $\norm{nx + y}<\psi(n)$ with infinitely
  many integers $n$. Does the same necessarily hold for all rational
  combinations of the $y_i$'s with $1$?
\end{question}

A `yes' would imply in particular that the homogeneous situation is
all we need to study. In other words, we can ask the following \emph{a
  priori} weaker question.

\begin{question}\label{q:hominhom}
  Suppose that almost every real $x$ satisfies $\norm{nx}<\psi(n)$
  with infinitely many integers $n\geq 1$. Does this imply that for
  any real $y$, almost every $x$ satisfies $\norm{nx+y}<\psi(n)$ with
  infinitely many integers $n\geq 1$? That is, if almost every real
  number is homogeneously $\psi$-approximable, does this imply that
  almost every real number is \emph{inhomogeneously}
  $\psi$-approximable with any inhomogeneous parameter?
\end{question}

It seems reasonable to think so. After all, it is already known to be
true for monotonic approximating functions, by Khintchine's Theorem
and its inhomogeneous analogue. Still, one may pursue the following
counter-question, perhaps in search of a contradiction.

\begin{question}\label{q:noyes}
  Let $\set{y_i}$ and $\set{z_i}$ be disjoint sequences of real
  numbers. Is there an approximating function $\psi$ such that
  \begin{itemize}
  \item for any $y\in\set{y_i}$, \textbf{almost no} real numbers $x$
    satisfy the inequality $\norm{nx+y}< \psi(n)$ with infinitely many
    integers $n\geq 1$, while
  \item for any $z\in\set{z_i}$, \textbf{(almost) all} real numbers
    $x$ satisfy the inequality $\norm{nx+z}< \psi(n)$ with infinitely
    many integers $n\geq 1$?
  \end{itemize}
\end{question}

\begin{remark*}
  Note that the second condition forces $\sum_n\psi(n)$ to diverge.
\end{remark*}

If the answers to Questions~\ref{q:ratcomb} and~\ref{q:hominhom} are
`yes'---and we suspect they are---we can still ask
Question~\ref{q:noyes} with the additional stipulation that every
$z\in\set{z_i}$ lies outside the rational span of $1$ and
$\set{y_i}$. This leads to Theorem~\ref{thm:moreover}, which we handle
in~\S\ref{sec:proof-theor-refthm:m}.

\section{Proof of Theorem~\ref{thm:moreover}: Dynamically defined
  covering systems}
\label{sec:proof-theor-refthm:m}

For the integers $T:=\set{1 < T_1 < \dots < T_M}$ and
$R:=\set{R_1, \dots, R_M}$, consider the \emph{system of residues}
\begin{equation*}
  \calS(T, R) := \set{n\in\ZZ : n \equiv
    R_m \;(\bmod\;T_m) \textrm{ for some } m=1, \dots, M}.
\end{equation*}
We say $\calS(T,R)$ is a \emph{covering system} if it contains
\emph{all} the integers. Erd\H{o}s asked the following question in
1950: \emph{Do there exist covering systems with arbitrarily large
  $T_1$?} He conjectured that there did exist such covering
systems~\cite{Erdoscongruences}, and even offered a cash prize for a
proof~\cite{Erdos2}. However, the conjecture was \emph{disproved} by
Hough in 2015; he showed that there is no covering system with
$T_1 > 10^{16}$~\cite[Theorem~1]{Hough}. The proof builds on work of
Filaseta, Ford, Konyagin, Pomerance, and Yu where among other things
they studied the asymptotic density of the \emph{complement} of a
given system of residues (\cite{FFKPY}, 2007).

Notice that a system of residues $\calS(T,R)$ (whether it is covering
or not) is a periodic subset of the integers. If we center an interval
$\parens*{-\frac{1}{2}, \frac{1}{2}}$ around each point of the system,
then the resulting union of intervals is also periodic. Then the
asymptotic density of the system of residues can be interpreted as the
proportion of the real line that is occupied by this union of
intervals. Since this proportion is invariant under scaling, we can
contract by a multiple of the period, and interpret the proportion of
the real line occupied by our union of intervals as exactly the
measure of the contracted set intersected with the unit interval
$[0,1)$.

We now ask a related question. Let us start with an $M$-tuple
$\bt := (t_1, \dots, t_M)$ of integers such that
$1\leq t_1 < \dots < t_M$, and a real number $\eps>0$, and let us
consider translations of $t_m\ZZ$ by \emph{any} real numbers
$\br:=(r_1, \dots, r_M)$. Again, the union of these is periodic.  How
much of the real line can we occupy by centering an interval
$\parens{-\frac{\eps}{2}, \frac{\eps}{2}}$ at each point of this set?
What we are really asking is: how much of $[0,1)$ is occupied by
\begin{equation}\label{eq:set}
  \bigcup_{m=1}^M\frac{t_m\ZZ + r_m}{K} + \parens*{-\frac{\eps}{2K}, \frac{\eps}{2K}}\cap [0,1),
\end{equation}
where $K = \prod t_i$ (or, alternatively,
$K=\operatorname{lcm}(t_1, \dots, t_M)$, or any multiple thereof)?
Hough's result suggests that if $t_1/\eps>10^{16}$, then the
set~\eqref{eq:set} cannot equal $[0,1)$. But what about the
\emph{measure} of the set~\eqref{eq:set}?  It will be convenient to
denote by $\mu(\eps, \bt, \br)$ the measure of its complement, and
define
\begin{equation*}
  \alpha(\eps, \bt) := \prod_{m=1}^M\parens*{1 - \frac{\eps}{t_m}}.
\end{equation*}
Now we are asking how \emph{small} we can make $\mu(\eps, \bt, \br)$.
For this, we can derive some answers from~\cite{FFKPY}. For example,
the following lemma shows that the expected value of $\mu$ is
$\alpha$.

\begin{lemma}[Version of~{\cite[Lemma~5.1]{FFKPY}} for real residues]\label{lem:filaetal}
  For any fixed $M$-tuple $\bt=(t_1, \dots, t_M)$ of integers with
  $1\leq t_1 < \dots < t_M$ and any $\eps>0$, the expected value of
  $\mu(\eps, \bt, \br)$ is $\alpha(\eps, \bt)$.
\end{lemma}

\begin{proof}[\textbf{Proof}]
  Letting
  \begin{equation*}
    E(m,r) := \frac{t_m\ZZ + r}{K}
    + \parens*{-\frac{\eps}{2K}, \frac{\eps}{2K}},
  \end{equation*}
  the expectation is the integral
  \begin{equation*}
    \frac{1}{K^M}\int_{[0,K)^M} \abs*{\widehat E(1,r_1)\cap\dots\cap \widehat E(M,r_M)}\,dr_1\,\dots\,dr_M = \prod_{m=1}^M\parens*{1 - \frac{\eps}{t_m}},
  \end{equation*}
  which proves the lemma.
\end{proof}

Moreover, Filaseta \emph{et al.} show (in the context of the original
problem of Erd\H{o}s) that deviations from this expected value are on
average extremely small for large $t_1$.

\begin{theorem}[{\cite[Theorem~7]{FFKPY}}]\label{thm:filavariance}
  For any fixed $M$-tuple $\bt = (t_1, \dots, t_M)$ of integers with
  $1\leq t_1 < \dots < t_M$ and $\eps = 1/k$ for some integer
  $k\geq1$, the variance of $\mu(\eps, \bt, \br)$ over integer values
  of $\br$ is
  $\ll \frac{\alpha(\eps, \bt)^2 \log (t_1/\eps)}{(t_1/\eps)^2}$.
\end{theorem}

\begin{remark*}
  It is a computation to verify that Theorem~\ref{thm:filavariance} is
  also true for general $\eps>0$ and with variance over all real
  $M$-tuples $\br$.
\end{remark*}

In the following we conjecture that given any irrational number $z$,
one can find positive integers $t_1<\dots < t_M$ for which the
$M$-tuple $\br = \bt z = (t_1 z, \dots, t_M z)$ is ``typical,'' in the
sense that $\mu(\eps, \bt, \bt z)$ is close to what
Lemma~\ref{lem:filaetal} and Theorem~\ref{thm:filavariance} lead us to
expect. Furthermore, we can choose the $t_m$'s from certain
dynamically defined positive-density subsets of the integers, and such
that $\alpha(\eps, \bt)$ is arbitrarily small. (Note that in the
notation, $n_m+1$ will play the role of $t_m$.)

\begin{conjecture}[Dynamically defined covering systems]\label{lem:conjlem}
  Suppose $y_1, \dots, y_\ell, z$ are real numbers such that $z$ is
  not in the rational span of $1, y_1, \dots, y_\ell$, and let
  $\eps, \delta>0$. Then for any $n_0>0$ there are positive integers
  $\set{n_0 <n_1 < n_2 < \dots < n_M}$ and real numbers
  $w_1, \dots, w_\ell$ such that the following hold:
  \begin{enumerate}
  \item We have $\max_i \norm{n_m y_i-w_i}\leq \eps/2$ for all
    $m = 1, \dots, M$.
  \item The measure of
    \begin{equation}\label{eq:conj}
      \parens*{\bigcup_{m=1}^M\frac{\ZZ+z}{k_m}
        + \parens*{-\frac{\eps}{2K}, \frac{\eps}{2K}}}\cap [0,1)
    \end{equation}
    is at least $1-\delta$, where $K=\prod_{m=1}^M(n_m+1)$ and
    $k_m = K/(n_m+1)$.
  \end{enumerate}
\end{conjecture}

There is good reason to believe Conjecture~\ref{lem:conjlem}. (In
fact, it is even reasonable to believe that the $w$'s are
superfluous.)  In the following lemma we show that we can find
arbitrarily many progressions $\set{n_1 < \dots < n_M}$, with $n_1$
arbitrarily large, such that the \emph{expected} measure
of~(\ref{eq:conj}), as calculated in Lemma~\ref{lem:filaetal}, is as
close to $1$ as we want. And Theorem~\ref{thm:filavariance} suggests
that the actual measure of~\eqref{eq:conj} is extremely likely to be
near its expected value.

\begin{lemma}
  Suppose $y_1, \dots, y_\ell$ are real numbers and $\eps, \delta >0$.
  Suppose that $(w_1, \dots, w_\ell)$ is an accumulation point for the
  orbit of $0$ by $(y_1, \dots, y_\ell)$-translations of the
  $\ell$-dimensional torus. Then for any $n_0>0$ there are integers
  $\set{n_0< n_1 < \dots < n_M}$ such that
  $\max_{i} \norm{n_m y_i - w_i}\leq \eps/2$ for all $m=1, \dots, M$,
  and such that
  $\prod_{m=1}^M\parens*{1 - \frac{\eps}{n_m+1}} < \delta$.
\end{lemma}

\begin{proof}[\textbf{Proof}]
  Let $\set{\bar n_m}$ be the sequence of all positive integers such
  that $\max_{i} \norm{\bar n_m y_i - w_i}\leq \eps/2$. We will show
  that $\set{\bar n_m}$ contains an infinite generalized arithmetic
  subsequence. That is, there is a set
  $S=\set{s_j : j=1, \dots, 2^{\ell}}$ of positive integers and an
  infinite subsequence $\set{n_m}\subset\set{\bar n_m}$ such that
  $n_{m+1} - n_m \in S$ for all $m$.

  We provide a simple construction. Let the $s_j$ be such that
  $\calO_j + s_j (y_1, \dots, y_\ell)$ lies within
  $B :=(w_1, \dots, w_\ell)+(-\eps/2, \eps/2)^\ell$, where $\calO_j$
  is the $j$th hyper-octant of $B$. Then the sequence $\set{n_m}$ can
  be constructed in the following way: Fix an initial term $n_1$ such
  that $n_1 (y_1, \dots, y_\ell)$ lies in $B$. The terms thereafter
  are set by the rule $n_m = n_{m-1}+s_j$ whenever
  $n_{m-1}(y_1, \dots, y_\ell)\in \calO_j$.

  Now a simple calculation shows that by making $M$ large we can make
  $\prod_{m=1}^M\parens*{1 - \frac{\eps}{n_m+1}}$ arbitrarily small.
\end{proof}

\begin{proof}[\textbf{Proof of Theorem~\ref{thm:moreover}}]
  Let $\set{y_i}, \set{z_i}$ be as in the theorem statement, and
  assume Conjecture~\ref{lem:conjlem} to be true. Let us define the
  new sequence
  \begin{equation*}
    \set{\bar z_i} = \set{z_1, z_2, z_1, z_2, z_3, z_1, z_2, z_3,
      z_4, \dots},
  \end{equation*}
  and let $\set{\delta_j}\subset\parens{0,1}$ be a sequence decreasing
  to $0$.
  
  Let $\set{n_1^{(1)}< \dots <n_{M_1}^{(1)}}$ and $w_1^{(1)}$ be the
  progression and number guaranteed by Conjecture~\ref{lem:conjlem},
  with $y_1, w_1^{(1)}, \bar z_1, 2^{-1}, \delta_1$ playing the roles
  of $y_1, w_1, z, \eps, \delta$ respectively. Let
  \begin{equation*}
    K_1 = \prod_{m=1}^{M_1} \parens*{n_m^{(1)}+1}.
  \end{equation*}  
  Inductively, let $\set{n_1^{(j)}< \dots <n_{M_j}^{(j)}}$ and
  $w_1^{(j)}, \dots, w_j^{(j)}$ be the progression and real numbers
  guaranteed by Conjecture~\ref{lem:conjlem}, with
  $y_1, \dots, y_j, w_1^{(j)}, \dots, w_j^{(j)}, \bar z_j, 2^{-j},
  \delta_j$
  playing the roles of
  $y_1, \dots, y_\ell, w_1, \dots, w_\ell, z, \eps, \delta$,
  respectively, such that $n_1^{(j)}\geq K_{j-1}$. Let
  \begin{equation*}
    K_j = \prod_{m=1}^{M_j} \parens*{n_m^{(j)}+1}.
  \end{equation*}

  We may now define $\psi$ by~(\ref{eq:def}), as in the proof of
  Theorem~\ref{thm:sequencecounterex}. Our construction guarantees
  that for every $z\in\set{z_i}$ we have
  $\abs{\limsup_{n}E_n^z(\psi)} = 1$. A simple modification of
  Claim~\ref{cl:containment} shows that for any $i$ we will have
  \begin{equation*}
    \limsup_{n\to\infty} E_n^{y_i}(\psi) \subset \limsup_{j\to\infty}
    E_{K_j}^{y_i+w_i^{(j)}}\parens*{2^{-j}}.
  \end{equation*}
  And the proof of Claim~\ref{cl:limsup} shows that
  \begin{equation*}
    \abs*{\limsup_{j\to\infty}
      E_{K_j}^{y_i+w_i^{(j)}}\parens*{2^{-j}}}=0,
  \end{equation*}
  so we have $\abs{\limsup_n E_n^y(\psi)}=0$ for any
  $y\in\set{y_i}$. This proves Theorem~\ref{thm:moreover}.
\end{proof}

\section{Proof of Theorem~\ref{thm:arbitrary}: Equidistributed
  inhomogeneous parameters}
\label{sec:count-inhom-doubly}

For every positive integer $m$ let $f(x,m)\geq 0$ be an integrable
function of the real variable $x$. Let
\begin{equation*}
  \underline\mu(f) := \liminf_{M\to\infty}\frac{1}{M}\sum_{m=1}^M\int_0^1
  f(x, m)\,dx
\end{equation*}
and
\begin{equation*}
  \overline\mu(f) := \limsup_{M\to\infty}\frac{1}{M}\sum_{m=1}^M\int_0^1
  f(x, m)\,dx.
\end{equation*}
If the two coincide, denote
$\mu(f)=\underline\mu(f) = \overline\mu(f)$. For a set $A$ of pairs
$(x,m)$ such that $A_{k} = \set{(x,m)\in A \mid m=k}$ is
measurable for all $k$, denote $\mu(A):=\mu(\bone_A)$, and similarly for
$\underline\mu(A)$ and $\overline\mu(A)$. Notice that we will always
have $\underline\mu(A)\leq \overline\mu(A)\leq 1$.

The goal of this section is to prove the following ``doubly metric''
statement, from which will follow Theorem~\ref{thm:arbitrary}.

\begin{theorem}\label{thm:doublymetric}
  Let $\set{y_m}$ be an equidistributed sequence $\bmod\; 1$. Suppose
  $\psi$ is an approximating function such that $\sum_n\psi(n)$
  diverges. Let $R>0$ and let $F$ denote the set of pairs $(x,m)$ for
  which the inequality $\norm*{nx +y_m}<\psi(n)$ has at least $R$
  integer solutions $n\geq 1$. Then $\mu(F)=1$.
\end{theorem}

\begin{remark*}
  Theorem~\ref{thm:doublymetric} should be compared with the Doubly
  Metric Inhomogeneous Khintchine Theorem, where it is shown under the
  same condition on $\psi$ that for almost every pair $(x,y)$ the
  inequality $\norm{nx + y}<\psi(n)$ has infinitely many integer
  solutions $n\geq 1$. Since we are taking the sequence $\set{y_m}$ to
  be equidistributed, we may naturally expect that for almost every
  fixed $x$, the $y_m$'s are a generic sampling from the corresponding
  full set of $y$'s. What we prove is similar but weaker. The proof
  follows~\cite[Page 121, Theorem II]{CasselsintrotoDA}.
\end{remark*}

The proof of Theorem~\ref{thm:doublymetric} is based on the following
analogue of the Paley--Zygmund Lemma.

\begin{lemma}\label{lem:pz}
  Suppose that for $f(x,m)\geq 0$, the quantities $\mu(f)$ and
  $\mu(f^2)$ exist and are finite. Suppose
  $\mu(f) \geq a\sqrt{\mu(f^2)}$ and $0\leq b\leq a$, and let
  \begin{equation*}
    A = \set*{(x,m): f(x,m)\geq b\sqrt{\mu(f^2)}}.
  \end{equation*}
  Then $\underline\mu(A)\geq (a-b)^2$.
\end{lemma}

\begin{proof}[\textbf{Proof}]
  For every $M$ we have
  \begin{multline*}
    \parens*{\frac{1}{M}\sum_{m=1}^M\int_0^1 \bone_A(x,m)f(x,m)\,dx}^2
    \\ \leq \parens*{\frac{1}{M}\sum_{m=1}^M\int_0^1 \bone_A(x,m)\,dx}
    \parens*{\frac{1}{M}\sum_{m=1}^M\int_0^1 f(x,m)^2\,dx},
  \end{multline*}
  by the Cauchy--Schwarz Inequality, and therefore in the limit as
  $M\to\infty$ we will have
  $\underline\mu\parens*{\bone_A\cdot f}_1^2 \leq \underline\mu(A)\,
  \mu\parens*{f^2}$.
  Now, since on the complement of $A$ we have
  $f(x,m)\leq b\sqrt{\mu(f^2)}$, we therefore have for every fixed $M$
  \begin{multline*}
    \frac{1}{M}\sum_{m=1}^M\int_0^1\bone_A(x,m)f(x,m)\,dx \\
    = \frac{1}{M}\sum_{m=1}^M\int_0^1f(x,m)\,dx -
    \frac{1}{M}\sum_{m=1}^M\int_0^1\bone_{A^c}(x,m)f(x,m)\,dx \\
    \geq \frac{1}{M}\sum_{m=1}^M\int_0^1f(x,m)\,dx - b\sqrt{\mu(f^2)},
  \end{multline*}
  and in the limit as $M\to\infty$ we will find
  $\underline\mu\parens*{\bone_A\cdot f} \geq \mu(f) -
  b\sqrt{\mu\parens*{f^2}} \geq (a-b)\sqrt{\mu\parens*{f^2}}$.
  Combining, we have that $\underline\mu(A)\geq (a-b)^2$.
\end{proof}

Now, let $\set{y_m}$ and $\psi$ be as in the statement of
Theorem~\ref{thm:doublymetric}, and let $\Delta_N(x,m)$ denote the
number of integer solutions of $\norm{nx+y_m}<\psi(n)$ with
$0<n\leq N$. Then
\begin{equation*}
  \Delta_N(x,m) = \sum_{n=1}^N \bone_n(nx + y_m) \quad\textrm{where}\quad\bone_n := \bone_{\parens*{-\frac{\psi(n)}{n}, \frac{\psi(n)}{n}}}.
\end{equation*}
Notice that
\begin{equation*}
  \mu\parens*{\Delta_N} = \sum_{n=1}^N 2\psi(n),
\end{equation*}
so our divergence assumption can be stated as
$\mu\parens*{\Delta_N}\to\infty$ as $N\to\infty$.

\begin{lemma}\label{lem:indep}
  For $n\neq k$, we will have
  \begin{equation*}
    \lim_{M\to\infty}\frac{1}{M}\sum_{m=1}^M\int_0^1\bone_n(nx+y_m)\bone_k(kx+y_m)\,dx
    = 4\psi(n)\psi(k).
  \end{equation*}
\end{lemma}

\begin{proof}[\textbf{Proof}]
  The function
  \begin{equation*}
    \int_0^1\bone_n(nx+y)\bone_k(kx+y)\,dx
  \end{equation*}
  is continuous in the real variable $y$. Therefore, since $\set{y_m}$
  is equidistributed, we have that
  \begin{equation*}
    \lim_{M\to\infty}\frac{1}{M}\sum_{m=1}^M\int_0^1\bone_n(nx+y_m)\bone_k(kx+y_m)\,dx
    = \int_0^1\int_0^1\bone_n(nx+y)\bone_k(kx+y)\,dx\,dy,
  \end{equation*}
  which is equal to $4\psi(n)\psi(k)$.
\end{proof}

\begin{corollary}\label{cor:asdf}
  For any $\eps>0$, we will have
  $\mu\parens*{\Delta_N} \geq
  (1-\eps)\sqrt{\mu\parens*{\Delta_N^2}}$
  for all sufficiently large $N$.
\end{corollary}

\begin{proof}[\textbf{Proof}]
  We calculate
  \begin{align*}
    \mu\parens*{\Delta_N^2}&= \lim_{M\to\infty}\frac{1}{M}\sum_{m=1}^M
                             \int_0^1\Delta_N(x,m)^2\,dx \\
                           &= \lim_{M\to\infty}\frac{1}{M}\sum_{m=1}^M
                             \int_0^1\sum_{n,k\leq N}\bone_n(nx + y_m)\bone_k(kx + y_m)\,dx\\
                           &= \lim_{M\to\infty}\frac{1}{M}\sum_{m=1}^M
                             \int_0^1\parens*{\sum_{n\leq N}\bone_n(nx + y_m)^2 +
                             \sum_{\substack{n, k\leq N \\ n\neq k}}\bone_n(nx +
    y_m)\bone_k(kx + y_m)}\,dx \\
                           &\overset{Lem.~\ref{lem:indep}}{=} \mu\parens*{\Delta_N} +
                             \sum_{\substack{n,k\leq N \\ n\neq k}}4\psi(n)\psi(k)\\
                           &\leq \mu\parens*{\Delta_N} + \mu\parens*{\Delta_N}^2 \\
                           &\leq (1-\eps)^{-2} \mu\parens*{\Delta_N}^2
  \end{align*}
  for $N$ sufficiently large, since $\mu\parens*{\Delta_N}\to\infty$
  as $N\to\infty$.
\end{proof}

\begin{proof}[\textbf{Proof of Theorem~\ref{thm:doublymetric}}]
  For an arbitrary small $\eps>0$, let $a=1-\eps$ and
  $b=\eps$. Corollary~\ref{cor:asdf} and our divergence assumption
  tell us that for $N$ sufficiently large we have
  \begin{equation*}
    \mu\parens*{\Delta_N} \geq (1-\eps)\sqrt{\mu\parens*{\Delta_N^2}} \quad\textrm{and}\quad  \eps\mu\parens*{\Delta_N} \geq R.
  \end{equation*}
  For these $N$, Lemma~\ref{lem:pz} implies that we will have
  $\Delta_N(x,m)\geq \eps\mu\parens*{\Delta_N} \geq R$ on a set $F_N$
  with $\underline \mu(F_N)\geq (1-2\eps)^2$. Notice that
  $F_N\subset F$. Since $\eps>0$ was arbitrary, this shows $\mu(F)=1$.
\end{proof}

Before stating the proof of Theorem~\ref{thm:arbitrary}, we prove the
following simple lemma.

\begin{lemma}\label{lem:simple}
  If $0\leq a_m\leq 1$ and
  $\lim_{M\to\infty}\frac{1}{M}\sum_{m=1}^M a_m=1$, then for every
  $\eps>0$, the set of integers $m\geq 1$ for which $a_m \geq 1-\eps$
  has asymptotic density $1$.
\end{lemma}

\begin{proof}[\textbf{Proof}]
  Suppose $0\leq a_m\leq 1$ and that there is some $\eps>0$ for which
  the set of integers $m\geq 1$ with $a_m \geq 1-\eps$ has lower
  asymptotic density $A<1$.  There is some sequence $M_j\to\infty$ of
  integers ``realizing'' this lower density, and so
  \begin{equation*}
    \frac{1}{M_j}\sum_{m=1}^{M_j}a_m < A + (1-A)(1-\eps) +
    o(1)\quad\textrm{as}\quad j\to\infty.
  \end{equation*}
  But $A + (1-A)(1-\eps) = 1 - \eps(1-A) < 1$, so
  $\lim_{M\to\infty}\frac{1}{M}\sum_{m=1}^M a_m\neq 1$.
\end{proof}

\begin{proof}[\textbf{Proof of Theorem~\ref{thm:arbitrary}}]
  Theorem~\ref{thm:doublymetric} tells us that $\mu(F)=1$, that is,
  \begin{equation*}
    \lim_{m\to\infty}\frac{1}{M}\sum_{m=1}^M\int_0^1\bone_F(m,x)\,dx = 1.
  \end{equation*}
  The theorem follows by applying Lemma~\ref{lem:simple} with
  $a_m = \int_0^1\bone_F(x,m)\,dx$.
\end{proof}

\section{Discussion of inhomogeneous versions of the Duffin--Schaeffer
  Conjecture}
\label{sec:inhom-duff-scha}

For the approximating functions $\psi$ coming from
Theorems~\ref{thm:sequencecounterex}/\ref{thm:moreover}
and~\ref{thm:cantorex}, notice that
\begin{align*}
  \sum_n \frac{\varphi(n)\psi(n)}{n} &= \sum_j \sum_{m=1}^{M_j}
                                       \frac{\varphi(k_m)
                                       k_m}{k_mK_j}2^{-j} \\
                                     &\leq \sum_j \frac{1}{2^jK_j}\sum_{d\mid K_j}\varphi(d) \\
                                     &= \sum_j 2^{-j},
\end{align*}
which converges. It therefore makes sense to formulate inhomogeneous
versions of the Duffin--Schaeffer Conjecture that take this into
account.

In~\S\ref{sec:inhomothy} we have stated the direct translation of the
Duffin--Schaeffer Conjecture to the inhomogeneous setting, the
\emph{Inhomogeneous Duffin--Schaeffer Conjecture}
(\textsc{idsc}). Obviously, it would be ideal to prove the
\textsc{idsc}, but since it is \emph{a priori} stronger than the
original conjecture, it may make more sense to aim for more modest
goals first. Can we prove the \textsc{idsc} for a particular
inhomogeneous parameter $y$?  Can we prove it for a family of
inhomogeneous parameters, perhaps badly approximable $y$'s? Can we
prove it for \emph{almost every} $y$? Or prove that it is a
``zero-one'' situation with respect to inhomogeneous parameters? In
the spirit of Question~\ref{q:hominhom}, does the $\textsc{dsc}$ imply
the $\textsc{idsc}$? Any of these would be nice.

In~\S\ref{sec:inhomothy} we remarked on the temptation to remove
monotonicity from the Inhomogeneous Khintchine Theorem by allowing the
inhomogeneous part to vary among a (say, equidistributed)
sequence. Theorem~\ref{thm:sequencecounterex} shows that this is
impossible, but Theorem~\ref{thm:arbitrary} shows that taking an
equidistributed sequence still results in a best-case scenario where
we can say \emph{something}. This leads us to the following question.

\begin{question}[Countably inhomogeneous
  Duffin--Schaeffer Conjecture]\label{q:cidsc}
  Let $\set{y_i}$ be some sequence of real numbers, and suppose $\psi$
  is an approximating function such that the sum
  $\sum_n \varphi(n)\psi(n)/n$ diverges. Does this imply that for
  almost every real number $x$ there exists an integer $m\geq 1$ such
  that infinitely many coprime integer pairs $(a,n)$ satisfy the
  inequality ${\abs{nx-a+y_m}<\psi(n)}$?
\end{question}

Since we are inclined to believe that the Inhomogeneous
Duffin--Schaeffer Conjecture is true, we must therefore also expect an
affirmative answer to Question~\ref{q:cidsc}, regardless of the given
sequence $\set{y_i}$. But in order to attack Question~\ref{q:cidsc}
directly, it makes more sense to restrict our attention to certain
kinds of sequences. In view of Theorem~\ref{thm:arbitrary}, the most
natural ones to consider are equidistributed.

In fact, one could modify Question~\ref{q:cidsc} in a number of
ways. Instead of a ``countably inhomogeneous'' version of the
Duffin--Schaeffer Conjecture, one may also seek such a version of the
Duffin--Schaeffer \emph{Theorem}, where we make the additional
assumption that
\begin{equation*}
  \limsup_{N\to\infty}\parens*{\sum_{n=1}^N
    \frac{\varphi(n)\psi(n)}{n}}\parens*{\sum_{n=1}^N\psi(n)}^{-1}>0. 
\end{equation*}
In any case, it would be helpful to have zero-one laws analogous to
Gallagher's. We provide some in~\S\ref{sec:inhom-zero-one}.

\section{Proof of Theorem~\ref{thm:E01}: Inhomogeneous zero-one laws}
\label{sec:inhom-zero-one}

The aim in this section is to prove the following restatement of
Theorem~\ref{thm:E01}.

\begin{theorem}[Theorem~\ref{thm:E01}]\label{thm:EbarE}
  Let $y$ be a real number. Let $\E^y$ denote the set of real numbers
  $x$ for which
  \begin{equation*}
    \abs{nx - a + y}<\psi(n), \quad (a,n)=1, 
  \end{equation*}
  for infinitely many integers $a,n$. Let $\E = \bigcup_{m}\E^{my}$
  and $\bar \E = \limsup_{m\to\infty}\E^{my}$. Then
  $\abs{\E} \in\set{0,1}$ and $\abs{\bar \E}\in\set{0,1}$.
\end{theorem}

\begin{remark*}
  $\E$ is the set of real numbers $x$ for which there exists an
  integer $m$ such that $\abs{nx - a + my}<\psi(n)$ has infinitely
  many solutions $(a,n)=1$, and $\bar \E$ is the set of real numbers
  $x$ for which there exist \emph{infinitely many} such integers
  $m\geq 1$.
\end{remark*}

The proof of Theorem~\ref{thm:EbarE} follows
Gallagher~\cite{Gallagher01}, which in turn relies partly on the
following lemma.

\begin{lemma}[Cassels,~\cite{Casselslemma}]\label{lem:cassels}
  Let $\set{I_k}$ be a sequence of intervals and let $\set{U_k}$ be a
  sequence of measurable sets such that for some positive $\eps <1$,
  \begin{equation*}
    U_k\subset I_k, \quad \abs*{U_k} \geq \eps\abs*{I_k}, \quad
    \abs*{I_k}\to 0.
  \end{equation*}
  Then
  $\abs*{\limsup_{k\to\infty}I_k} = \abs*{\limsup_{k\to\infty}U_k}$.
\end{lemma}

\begin{proof}[\textbf{Proof of Theorem~\ref{thm:EbarE}}]
  This proof follows Gallagher~\cite{Gallagher01}. The first reduction
  is to the case $\psi(n)=o(n)$, justified by the following fact:
  \begin{quote}
    \emph{The length $L_n$ of the longest interval of consecutive
      integers not coprime to $n$ satisfies $L_n = o(n)$.}
  \end{quote}
  In the proof we will express $\E=A\cup B \cup C$, and show that
  there is a $0$-$1$ law for each of the sets $A, B, C$. For this, we
  will show that each of these is invariant under certain ergodic
  transformations.

  For a prime number $p$ and integers $m,\nu$ with $\nu \geq 1$,
  consider the inequality
  \begin{equation}\label{eq:star}
    \abs{nx -a +my} < p^{\nu-1}\psi(n), \quad (a,n)=1.
  \end{equation}
  We define the sets
  \begin{align*}
    A(p^\nu, m) &= \set*{x \textrm{ satisfying~\eqref{eq:star}
                  infinitely often with } p\nmid n }\\
    B(p^\nu, m) &= \set*{x \textrm{ satisfying~\eqref{eq:star}
                  infinitely often with } p\parallel n } \\
    C(p^\nu, m) &= \set*{x \textrm{ satisfying~\eqref{eq:star}
                  infinitely often with } p^2\mid n },
  \end{align*}
  and
  \begin{align*}
    A(p^\nu) &= \bigcup_m A(p^\nu, m) &\bar A(p^\nu) &= \limsup_{m\to\infty} A(p^\nu, m) \\
    B(p^\nu) &= \bigcup_m B(p^\nu, m) &   \bar B(p^\nu) &= \limsup_{m\to\infty} B(p^\nu, m) \\
    C(p^\nu) &= \bigcup_m C(p^\nu, m) & \bar C(p^\nu) &=
                                                        \limsup_{m\to\infty} C(p^\nu, m).
  \end{align*}
  Notice that $\E^{my} = A(p,m)\cup B(p,m) \cup C(p,m)$ for any prime
  $p$, and therefore that $\E = A(p)\cup B(p) \cup C(p)$ for any prime
  $p$. Notice also that
  $\bar \E = \bar A(p)\cup \bar B(p)\cup \bar C(p)$ for any prime $p$.

  It is clear that $A(p,m)\subseteq A(p^\nu, m)$ for any $\nu\geq 1$.
  Since we are assuming that $\psi(n)=o(n)$, we may use
  Lemma~\ref{lem:cassels} to conclude that
  $\abs*{A(p,m)}=\abs*{A(p^\nu,m)}$. It is therefore clear that
  $A(p)\subseteq A(p^\nu)$ and $\abs*{A(p)} = \abs*{A(p^\nu)}$, and
  therefore that $\bigcup_{\nu\geq 1}A(p^\nu)$ has the same
  measure. Also, we have that $\bar A(p)\subseteq \bar A(p^\nu)$, and
  the Borel--Cantelli Lemma implies that
  $\abs*{\bar A(p)} = \abs*{\bar A(p^\nu)}$. Therefore,
  $\bigcup_{\nu\geq 1}\bar A(p^\nu)$ has the same measure. This
  paragraph holds also after replacing $A$'s with $B$'s.

  (In the remaining paragraphs of this proof, all instances of
  $A, B, C, \E$ can be replaced with $\bar A, \bar B, \bar C, \bar \E$
  to prove the $0$-$1$ law for $\bar \E$.)

  Notice that if $x$ satisfies~\eqref{eq:star} for some $m$ and
  $p\nmid n$, then
  \begin{equation*}
    \abs{n(px) - pa + mpy} < p^\nu\psi(n), \quad (pa, n)=1,
  \end{equation*}
  which shows that multiplication by $p$ carries $A(p^\nu)$ into
  $A(p^{\nu+1})$, and therefore $\bigcup_{\nu\geq 1} A(p)$ is taken
  into itself. Since multiplication by $p$ is an ergodic
  transformation of the circle, $\bigcup_{\nu\geq 1} A(p)$ must
  therefore have measure $0$ or $1$, hence $A(p)$ has measure $0$ or
  $1$.

  As for $B$, notice that if $x$ satisfies~\eqref{eq:star} with some
  $m$ and $p\parallel n$, then
  \begin{equation*}
    \abs*{n\parens*{px+\frac{1}{p}} - pa - \frac{n}{p} + mpy}< p^\nu
    \psi(n), \quad \parens*{pa + \frac{n}{p}, n}=1,
  \end{equation*}
  and the same arguments will show that $B(p)$ has measure $0$ or $1$,
  this time using that $x\mapsto px + \frac{1}{p}$ is ergodic.

  Now we know that if either $A(p)$ or $B(p)$ have positive measure,
  then $\E$ is full. So let us assume that $\abs{A(p)}=\abs{B(p)}=0$
  for all $p$, so that $\abs{\E}=\abs{C(p)}$ and in fact
  $\abs{\E\triangle C(p)}=0$ for all $p$. If $m,a,n$
  satisfy~\eqref{eq:star} with $p^2\mid n$ and $\nu=1$, then
  \begin{equation*}
    \abs*{n\parens*{x \pm \frac{1}{p}} - a \pm \frac{n}{p} + my} <
    \psi(n),\quad \parens*{a\pm \frac{n}{p}, n}=1,
  \end{equation*}
  which shows that $C(p)$ is periodic with period $1/p$. This means in
  particular that if $I$ is any interval of length $1/p$, then
  $\abs{C(p)\cap I} = \abs{C(p)}\cdot\abs{I}$. And since
  $\abs{\E\triangle C(p)}=0$, we have
  $\abs{\E\cap I} = \abs{\E}\cdot\abs{I}$.

  Now suppose that $\abs{\E}>0$, and let $x_0$ be a density point of
  $\E$. Let
  \begin{equation*}
    I_p = \parens*{x_0-\frac{1}{2p}, x_0 + \frac{1}{2p}}.
  \end{equation*}
  Then by Lebesgue's density theorem $\abs{\E\cap I_p}\sim \abs{I_p}$
  as $p\to\infty$. Therefore $\abs{\E}=1$, finishing the proof. (Also,
  for $\bar \E$, after taking the above parenthetical into account.)
\end{proof}

By inspecting the proof of Theorem~\ref{thm:EbarE} we can deduce the
following.

\begin{theorem}[Theorem~\ref{thm:eitheror}]
  Let $y$ be a real number. Let $\E^y$ denote the set of real numbers
  $x$ for which
  \begin{equation*}
    \abs{nx - a + y}<\psi(n), \quad (a,n)=1, 
  \end{equation*}
  for infinitely many integers, $a,n$. Then at least one of the
  following holds:
  \begin{itemize}
  \item $\E^{my}$ is null for every $m$.
  \item There is some $m$ for which $\E^{my}$ is full.
  \item For any $\eps>0$ there are arbitrarily many $m\geq 1$ with
    $\abs{\E^{my}}>1-\eps$.
  \end{itemize}
\end{theorem}

\begin{proof}[\textbf{Proof}]
  By Theorem~\ref{thm:E01} we know that the measure of
  $\E=\bigcup_m \E^{my}$ is either $0$ or $1$. If it is $0$, then of
  course $\E^{my}$ is also null for every $m$. So suppose that
  $\abs{\E}=1$. This implies that some $\E^{my}$ must have positive
  measure. But $\E^{my} = A(p,m)\cup B(p,m)\cup C(p,m)$ for any prime
  $p$, so one of these three sets must have positive measure. If for
  infinitely many primes $p$ we have $\abs{A(p,m)}=\abs{B(p,m)}=0$,
  then our argument in the proof of Theorem~\ref{thm:E01} shows that
  $\E^{my}$ must be full. On the other hand if $\abs{A(p,m)}>0$, then
  $\abs{p^\nu A(p,m)} > 1-\eps$ if $\nu$ is large enough. And our
  arguments in the previous proof show that
  $\abs{p^\nu A(p,m)} \leq \abs{A(p^{\nu+1},p^\nu m)} = \abs{A(p,
    p^\nu m)}$.
  Therefore, $\abs{\E^{p^{\nu} m y}}>1-\eps$ for all $\nu$ large
  enough.
\end{proof}

\subsection*{Acknowledgments}

This project was inspired and encouraged by numerous conversations
with Sanju Velani. I am grateful to both him and Victor Beresnevich
for welcoming me into the Number Theory Study Group at the University
of York, UK, and for initiating me into the field of metric
Diophantine approximation. I also thank the anonymous referee for a
careful reading.


\bibliographystyle{amsalpha}

\bibliography{../bibliography}

\end{document}